\documentclass[11pt,letterpaper]{amsart}

\usepackage[dvipsnames]{xcolor}
\usepackage[a4paper,top=3cm,bottom=2.5cm,left=3cm,right=3cm,marginparwidth=1.75cm]{geometry}

\usepackage{amsmath, amssymb, amsfonts, latexsym, mdwlist, amsthm}
\usepackage{subfig}
\usepackage{graphicx}
\usepackage{tikz-cd}
\usepackage{wrapfig}
\usepackage{mathrsfs}
\tikzstyle{startstop} = [rectangle, rounded corners, minimum width=3cm, minimum height=1cm,text centered, draw=black, fill=red!30]
\tikzstyle{io} = [trapezium, trapezium left angle=70, trapezium right angle=110, minimum width=3cm, minimum height=1cm, text centered, draw=black, fill=blue!30]
\tikzstyle{process} = [rectangle, minimum width=3cm, minimum height=1cm, text centered, draw=black, fill=orange!30]
\tikzstyle{decision} = [diamond, minimum width=3cm, minimum height=1cm, text centered, draw=black, fill=green!30]
\tikzstyle{arrow} = [thick,->,>=stealth]

\usepackage{mathtools}
\usepackage{tikz}
\usepackage{comment}
\usepackage[new]{old-arrows}

\usetikzlibrary{patterns}

\usepackage{enumerate}
\usepackage[pagebackref]{hyperref}

\usepackage{amsmath, amssymb, amsfonts, latexsym, mdwlist, amsthm}
\usepackage{subfig}
\usepackage{wrapfig}

\usepackage{mathtools}
\usepackage{bm}

\definecolor{lightred}{HTML}{ff4d4d}
\definecolor{lightblue}{HTML}{1F88CD}
\definecolor{lightgrey}{HTML}{727272}
\definecolor{lightblue2}{HTML}{009EC1}
\definecolor{mypink}{HTML}{FD00B0}

\usepackage{tikz}
\usetikzlibrary{calc,trees,positioning,arrows,chains,shapes.geometric,%
    decorations.pathreplacing,decorations.pathmorphing,shapes,%
    matrix,shapes.symbols}

\tikzset{
>=stealth',
  punktchain/.style={
    rectangle,
    rounded corners,
    draw=black, thick,
    minimum height=3em,
    text centered,
    on chain},
  line/.style={draw, thick, <-},
  element/.style={
    tape,
    top color=white,
    bottom color=blue!50!black!60!,
    minimum width=8em,
    draw=blue!40!black!90, very thick,
    text width=10em,
    minimum height=3.5em,
    text centered,
    on chain},
  every join/.style={->, thick,shorten >=1pt},
  decoration={brace},
  tuborg/.style={decorate},
  tubnode/.style={midway, right=2pt},
}

\usepackage{paralist}
\setdefaultenum{(a)}{(i)}{}{}
\usepackage[shortlabels]{enumitem} 
\usepackage{derivative}
\usepackage{bbm}

\makeatletter
\newtheorem*{rep@theorem}{\rep@title}
\newcommand{\newreptheorem}[2]{%
\newenvironment{rep#1}[1]{%
 \def\rep@title{#2 \ref{##1}}%
 \begin{rep@theorem}}%
 {\end{rep@theorem}}}
\makeatother

\newtheorem{theorem}{Theorem}[section]
\newreptheorem{theorem}{Theorem}
\newtheorem{proposition}[theorem]{Proposition}

\newtheorem{lemma}[theorem]{Lemma}

\newtheorem{corollary}[theorem]{Corollary}
\newreptheorem{corollary}{Corollary}

\newreptheorem{conjecture}{Conjecture}

\newtheorem{thm-int}{Theorem}

\theoremstyle{definition}
\newtheorem{Def-s}[theorem]{Definition}
\newtheorem{definition}[theorem]{Definition}

\newcommand{\ignore}[1]{}

\usepackage[numbers]{natbib}
\newcommand{\lra}{\longrightarrow}
\newcommand{\wt}{\widetilde}

\newcommand{\ZZ}{\mathbb{Z}}

\newcommand{\QQ}{\mathbb{Q}}
\newcommand{\RR}{\mathbb{R}}

\newcommand{\CC}{\mathbb{C}}

\newcommand{\hvol}{\widehat{\operatorname{vol}}}

\newcommand{\Val}{\mathrm{Val}}

\DeclareMathOperator{\Aut}{Aut}

\DeclareMathOperator{\coeff}{coeff}

\DeclareMathOperator{\vol}{vol}

\DeclareMathOperator{\ord}{ord}

\DeclareMathOperator{\Spec}{Spec}

\newcommand{\cX}{\mathcal{X}}
\newcommand{\cY}{\mathcal{Y}}

\newcommand{\cO}{\mathcal{O}}

\newcommand{\TT}{\mathbb{T}}

\newcommand{\bA}{\mathbb{A}}

\newcommand{\Gm}{\mathbb{G}_m}

\def\citestacks#1{\cite[\href{https://stacks.math.columbia.edu/tag/#1}{#1}]{stacks-project}}

\usepackage{hyperref}
\hypersetup{
	colorlinks=true,
    linkcolor={blue},
    citecolor={blue},
	urlcolor={black}
}

\begin{document}

\title[The finite degree formula for normalized volumes]{The finite degree formula for normalized volumes}

\author{Zhiyu Liu}

\address{Department of Mathematics, Princeton University, Princeton, NJ 08544, USA}
\email{zl3301@princeton.edu}
\urladdr{sites.google.com/view/zhiyuliu}

\begin{abstract}
Let $f\colon \big(x\in (X, \Delta_X)\big)\to \big(y\in (Y, \Delta_Y)\big)$ be a finite surjective morphism between klt singularities such that $K_X+\Delta_X=f^*(K_Y+\Delta_Y)$. We show that the normalized volumes satisfy
\[\widehat{\mathrm{vol}}(x, X, \Delta_X)=\mathrm{deg}(f)\cdot \widehat{\mathrm{vol}}(y, Y, \Delta_Y).\]
This proves a conjecture in \cite{li-liu-xu:guide,zhuang:survey,xu-zhuang:prob-list}.
\end{abstract}

\vspace{-1em}
\maketitle

\setcounter{tocdepth}{1}

\section{Introduction}

The normalized volume $\hvol(x,X,\Delta)$, introduced by Li in \cite{li:normalized-volume}, is a local
invariant of a klt singularity $x\in (X,\Delta)$ (see Definition~\ref{def:local-vol}). It plays a central role in the study of K-stability for Fano varieties \cite{xu:quasi-monomial,li-wang-xu:alg-cone,li-liu:ke-mini,li:equivariant-mini,li-xu:kollar-component,blum-xu:unique,blum-dhl-liu-xu:proper}, the local stability theory of singularities \cite{xu-zhuang:stable-deg,xu-zhuang:kaledin,xu-zhuang:boundedness}, and the moduli problems for explicit varieties \cite{liu-xu:cubic3,liu:cubic4}. We refer to \cite{li-liu-xu:guide,zhuang:survey,xu-zhuang:prob-list} for further background and details.

A powerful tool for studying normalized volumes is the \emph{finite degree formula}. More precisely, given a finite surjective morphism $f\colon \big(x\in (X, \Delta_X)\big)\to \big(y\in (Y, \Delta_Y)\big)$ between klt singularities satisfying $K_X+\Delta_X=f^*(K_Y+\Delta_Y)$, one may expect an equality 
\[\hvol(x, X, \Delta_X)=\mathrm{deg}(f)\cdot \hvol(y, Y, \Delta_Y).\]
When $f$ is a Galois cover, such a formula is proven in \cite{xu-zhuang:unique}. Since then, it has found many applications, including the moduli theory of varieties and boundedness problems~\cite{xu-zhuang:boundedness,han-qi-zhuang:boundedness,liu:optimal,blum-dhl-liu-xu:proper,blum-liu:lower,liu-xu:cubic3,liu:cubic4,han-liu-qi:acc,liu-zhao:fano3,liu-jh:boundedness-stable-family}. Despite its importance, the formula beyond the Galois case, which is predicted by \cite{li-liu-xu:guide,zhuang:survey,xu-zhuang:prob-list}, remains open.

In this paper, we prove the finite degree formula in full generality, as conjectured in \cite[Conjecture 6.4]{li-liu-xu:guide}, \cite[Conjecture 6.17]{zhuang:survey}, and \cite{xu-zhuang:prob-list}.

\begin{theorem}\label{thm:main}
Let $$f\colon \big(x\in (X, \Delta_X)\big)\lra \big(y\in (Y, \Delta_Y)\big)$$ be a finite surjective morphism between klt singularities. If $K_X+\Delta_X=f^*(K_Y+\Delta_Y)$, then
\[\hvol(x, X, \Delta_X)=\mathrm{deg}(f)\cdot \hvol(y, Y, \Delta_Y).\]
\end{theorem}

Note that its proof, which will be explained later, relies on an analytic input \cite[Theorem 1.2]{liu-zhu:equivariant}. Unlike the Galois case in \cite{zhuang:optimal,xu-zhuang:unique}, a purely algebraic proof is still unknown, which is sought by \cite{li-liu-xu:guide,zhuang:survey,xu-zhuang:prob-list}.

The above theorem is stated for (germs of) klt singularities. In the global setting, we have the following corollary.

\begin{corollary}\label{cor:global}
Let $f\colon (X, \Delta_X)\to (Y, \Delta_Y)$ be a finite surjective morphism between klt pairs. If $K_X+\Delta_X=f^*(K_Y+\Delta_Y)$, then for any closed point $y\in Y$, we have
\[\sum_{x_i\in f^{-1}(y)}\hvol(x_i, X, \Delta_X)=\mathrm{deg}(f)\cdot \hvol(y, Y, \Delta_Y).\]
\end{corollary}



\subsection*{Sketch of the proof}

Under the additional assumption that $f$ is Galois, Theorem~\ref{thm:main} is proved by \cite{xu-zhuang:unique}. The idea is to use the uniqueness of the minimizer of the normalized volume function on $X$ to show that any minimizer is invariant under the action of the Galois group of $f$. Therefore, it descends to a suitable valuation on $Y$ and gives the desired equality.

For the general case, \cite{zhuang:survey} suggests passing to the Galois closure of $f$. However, this may produce non-effective boundary divisors, and we need to study normalized volumes of sub-klt singularities. Unfortunately, in the sub-klt case, the uniqueness of minimizers can fail, and minimizers need not be invariant.

In this paper, we instead use a completely different method, by applying the stable degeneration theory developed in \cite{xu-zhuang:stable-deg,xu:quasi-monomial,li-xu:higher-rank} to reduce the problem to log Fano cones. 

More precisely, \cite{xu-zhuang:stable-deg} gives a degeneration of $y\in(Y,\Delta_Y)$ to a K-semistable log Fano cone $y_0\in (Y_0, \Delta_{Y_0})$; see Theorem~\ref{thm:stab-deg}. In Proposition~\ref{prop:pull-back-weak}, we show that up to a finite base change, such a degeneration can be lifted to a degeneration of $x\in (X, \Delta_X)$ to $x_0\in (X_0, \Delta_{X_0})$, with a finite log-crepant surjective morphism 
$$g_0\colon \big(x_0\in (X_0,\Delta_{X_0})\big)\longrightarrow \big(y_0\in (Y_0,\Delta_{Y_0})\big).$$ 
Then, we prove that the log Fano cone structure can also be lifted from $Y_0$ to $X_0$ such that the corresponding tori are related by an isogeny and $g_0$ is equivariant (cf.~Proposition~\ref{prop:lift-cone}).

Next, we want to identify $\hvol(x_0,X_0,\Delta_{X_0})$ with $\deg(f)\cdot \hvol(y_0,Y_0,\Delta_{Y_0})$. Building on approximation techniques in \cite{xu-zhuang:kaledin}, we prove a local analog of \cite[Theorem~1.2(1)]{liu-zhu:equivariant} in Theorem~\ref{thm:Kss} by reducing to the quasi-regular case proved in \cite{liu-zhu:equivariant}. This implies that $x_0\in (X_0, \Delta_{X_0})$ is also K-semistable if we choose a Reeb vector that is compatible with the one on $Y_0$. Using this, a calculation in Lemma~\ref{lem:vol-cone} gives 
\[\hvol(x_0,X_0,\Delta_{X_0})=\deg(f)\cdot \hvol(y_0,Y_0,\Delta_{Y_0}).\]

Now, applying the lower semicontinuity of normalized volumes in families \cite{blum-liu:lower}, we get an inequality $\hvol(x,X,\Delta_{X})\geq \hvol(x_0,X_0,\Delta_{X_0})$, while $\hvol(y,Y,\Delta_{Y})= \hvol(y_0,Y_0,\Delta_{Y_0})$ follows from the property of stable degeneration. Combining these results, we conclude that
\[\hvol(x, X, \Delta_X)\geq \mathrm{deg}(f)\cdot \hvol(y, Y, \Delta_Y).\]
The reverse inequality follows from a straightforward computation.

\subsection*{Acknowledgments}

I thank Chen Jiang, Yuchen Liu, Lu Qi, Yongbin Ruan, and Chenyang Xu for helpful discussions and comments.

\section{Preliminaries}

Throughout this paper, all schemes are essentially of finite type over the field $\CC$ of complex numbers and all morphisms are over $\CC$. A variety is an integral separated scheme of finite type over $\CC$. We always denote by $t$ the coordinate of the affine line $\bA^1$ over $\CC$. The function field of an integral scheme $X$ and the fraction field of an integral domain $A$ are denoted by $\CC(X)$ and $\CC(A)$, respectively. If $X$ is affine, we denote by $\cO(X)$ its coordinate ring.

A pair $(X, \Delta)$ consists of a normal integral scheme $X$ essentially of finite type over $\CC$ and an effective $\QQ$-divisor $\Delta$ on $X$ such that $K_X+\Delta$ is $\QQ$-Cartier. We follow the standard terminology for singularities of pairs as in \cite[Definition 2.8]{Kollar13}. A log Fano pair is a klt pair $(X, \Delta)$ such that $X$ is a normal projective variety and $-(K_X+\Delta)$ is ample.

A singularity $x\in (X, \Delta)$ is a pair $(X, \Delta)$ such that $X$ is the spectrum of a local ring essentially of finite type over $\CC$ and $x$ is the unique closed point with $\kappa(x)=\CC$. A singularity $x\in (X, \Delta)$ is called klt if $(X, \Delta)$ is klt.

\subsection{Valuations and normalized volumes}

For an integral scheme $X$, we denote by $\Val_{X}$ the set of all (real) valuations on $\CC(X)$ centered on $X$. We write $$\Val_{X, x}\subset \Val_{X}$$ for the set of valuations centered at a point $x\in X$. If $G$ is a group that acts on $X$ and fixes $x$, then we denote by $\Val_{X, x}^G$ the set of $G$-invariant valuations. The following notion is standard.

\begin{definition}\label{def:volume}
For a singularity $x\in X=\Spec(R)$ and any $v\in \Val_{X, x}$, we define an ideal of $R$ for each $m\in \RR$ by
\[\mathfrak{a}_m(v)\coloneqq \{r\in R\mid v(r)\geq m\}.\]

Then the \emph{volume} of $v$ is defined by
\[\vol_{X, x}(v)\coloneqq \limsup_{m\to \infty} \frac{n!}{m^n}\mathrm{length}(R/\mathfrak{a}_m(v)),\]
where $n=\dim X$.
\end{definition}
By \cite{ELS,mustata:convex-body,cut}, the above limsup is actually a limit.

We define the graded ring associated to $v$ by
\[\mathrm{gr}_v(R)\coloneqq \bigoplus_{\lambda\in \Phi_v} \frac{\mathfrak{a}_\lambda(v)}{\mathfrak{a}_{>\lambda}(v)},\]
where $\Phi_v$ is the value semigroup of $v$ and
\[\mathfrak{a}_{>\lambda}(v)\coloneqq \{r\in R\mid v(r)>\lambda\}.\]
If $v=\mathrm{ord}_E$ for a divisor $E$ over $X$, we write $\mathrm{gr}_E(R)\coloneqq \mathrm{gr}_{\mathrm{ord}_E}(R)$.

By virtue of \cite{boucksom:val-mult,mustata:val-asymp}, for any pair $(X, \Delta)$, there is a \emph{log discrepancy function of valuations}
\[A_{X, \Delta}\colon \Val_X\longrightarrow \RR\cup\{+\infty\},\]
which generalizes the classical definition of log discrepancies along divisors. In particular, if $E$ is a prime divisor over $X$, then $A_{X,\Delta}(\ord_E)=A_{X,\Delta}(E)$.

The following invariant of a klt singularity, which was first defined in \cite{li:normalized-volume}, plays a key role in our paper.

\begin{definition}\label{def:local-vol}
For a klt singularity $x\in (X, \Delta)$ and any $v\in \Val_{X, x}$, we define the \emph{normalized volume} of $v$ by
\[\hvol_{(X, \Delta)}(v)\coloneqq (A_{X,\Delta}(v))^{\dim X}\cdot \vol_{X,x}(v)\]
if $A_{X,\Delta}(v)<+\infty$, and set it equal to $+\infty$ otherwise. 

The \emph{normalized volume of a klt singularity} $x\in (X, \Delta)$ is
\[\hvol(x, X, \Delta)\coloneqq \inf_{\substack{v\in \Val_{X,x}\\ A_{X,\Delta}(v)<+\infty}} \hvol_{(X, \Delta)}(v).\]
\end{definition}

By \cite{blum:minimizer}, the function $\hvol_{(X, \Delta)}(-)$ has a minimizer $v$ on $\Val_{X,x}$, i.e.
\[\hvol_{(X, \Delta)}(v)=\hvol(x, X, \Delta).\]
Moreover, such a minimizer $v$ is unique up to scaling \cite{xu-zhuang:unique,blum-liu-qi:convex}, and is quasi-monomial \cite{xu:quasi-monomial}.

\subsection{Special degeneration}\label{subsec:special-deg}

In our paper, we focus on the following important class of one-parameter degenerations of klt singularities.

\begin{definition}
A \emph{special degeneration} $$\pi_X\colon (\cX, \Delta_{\cX}; \sigma_{\cX})\to \bA^1$$ of a klt singularity $x\in (X, \Delta)$ consists of 

\begin{itemize}
    \item a normal integral affine scheme $\cX$ with a $\Gm$-action,

    \item a flat surjective $\Gm$-equivariant morphism $\pi_X\colon \cX\to \bA^1$ with each fiber integral and normal, where $\Gm$ acts on $\bA^1$ by scaling, 

    \item a $\Gm$-equivariant section $\sigma_{\cX}$ of $\pi_X$,

    \item an effective $\QQ$-divisor $\Delta_{\cX}$ on $\cX$ such that $(\cX, \Delta_{\cX}+X_0)$ is plt, where $X_0\coloneqq \pi_X^{-1}(0)$, and

    \item a $\Gm$-equivariant isomorphism $$(\cX, \Delta_{\cX})|_{\bA^1\setminus \{0\}}\cong (X, \Delta)\times (\bA^1\setminus \{0\})$$ over $\bA^1\setminus \{0\}$ and $\sigma_{\cX}(t)=(x,t)$ for any $t\in \bA^1\setminus \{0\}$ under this isomorphism.
\end{itemize}

\end{definition}

In particular, $\mathrm{Supp}(\Delta_{\cX})$ does not contain $X_0$, and $(X_0, \Delta_{X_0})$ is a klt pair by adjunction, where $$\Delta_{X_0}\coloneqq (K_{\cX}+\Delta_{\cX}+X_0)|_{X_0}-K_{X_0}.$$ We also write $x_0\coloneqq \sigma_{\cX}(0)\in X_0$. Note that $\Delta_{\cX}$ is unique in the above definition, since it is the closure of $\Delta\times (\bA^1\setminus \{0\})$ in $\cX$.

In practice, we mainly work with special degenerations that are induced by certain divisorial valuations.

\begin{definition}
Let $x\in (X, \Delta)$ be a klt singularity. A prime
divisor $E$ over $(X, \Delta)$ is a \emph{Koll\'ar component over $x\in (X, \Delta)$} if there is a proper birational morphism $\pi\colon Y\to X$ such that $\pi$ is an isomorphism over $X\setminus \{x\}$, $E$ is the unique $\pi$-exceptional prime divisor on $Y$, $\pi(E)=x$, $-E$
is $\QQ$-Cartier and $\pi$-ample, and $(Y, \pi^{-1}_*\Delta+E)$ is plt. The map $\pi\colon Y\to X$ is called the \emph{plt blow-up} associated to $E$. By adjunction, we may write
\[(K_Y+\pi^{-1}_*\Delta+E)|_E=K_E+\Delta_E,\]
so that $(E,\Delta_E)$ is a log Fano pair.
\end{definition}

By \cite{xu:finite-fund-group}, every klt singularity has a Koll\'ar component. Moreover, \cite[Theorem 1.3]{li-xu:kollar-component} gives
\begin{equation}\label{eq:hvol-divisor}
\hvol(x,X,\Delta)=\inf_{\substack{v\in \Val_{X,x} \text{ is divisorial}}} \hvol_{(X, \Delta)}(v).
\end{equation}

Starting with a Koll\'ar component $E$ over $x\in (X, \Delta)$, we may construct a special degeneration as follows. We first define 
$$\cX\coloneqq \Spec_X\left(\bigoplus_{m\in \ZZ} \mathfrak{a}_m t^{-m} \right)$$
where $\mathfrak{a}_m\coloneqq \pi_*\cO_Y(-mE)$ when $m\geq 0$, and $\mathfrak{a}_m\coloneqq \cO_X$ otherwise. Then we have a natural $\Gm$-equivariant morphism $\pi_X\colon \cX\to \bA^1$ with a section $\sigma_{\cX}$. As discussed in \cite{li-xu:kollar-component,li-xu:higher-rank,li-liu-xu:guide}, there is a $\QQ$-divisor $\Delta_{\cX}$ on $\cX$ so that $\pi_X\colon (\cX, \Delta_{\cX}; \sigma_{\cX})\to \bA^1$ is a special degeneration of $x\in (X, \Delta)$ with the central fiber 
\[X_0=\Spec(\mathrm{gr}_{E}(\cO(X))).\]

\subsection{Log Fano cones}\label{subsec:cone}

Let $X$ be a normal variety and $\TT_X$ be a torus acting on $X$. We write $M_X\coloneqq \mathrm{Hom}(\TT_X,\Gm)$ for its character lattice and $N_X\coloneqq \mathrm{Hom}(\Gm, \TT_X)$ for its cocharacter lattice. For any $\alpha\in M_X$, we write $\chi^{\alpha}\colon \TT_X\to \Gm$ for the corresponding homomorphism. 

Set $\mathfrak{t}_{X,\mathbb{R}}\coloneqq N_X\otimes_{\ZZ} \mathbb{R}$ and $M_{X,\mathbb{R}}\coloneqq M_X \otimes_{\ZZ} \mathbb{R}$. Then we have a natural pairing
\[\langle -, -\rangle\colon M_{X,\mathbb{R}}\times \mathfrak{t}_{X,\mathbb{R}}\to \mathbb{R}.\]
If $X=\Spec(R)$, we have a weight decomposition
\[R=\bigoplus_{\alpha\in M_X} R_{\alpha}.\]
Define
\[\Gamma_X\coloneqq \{\alpha\in M_X\mid R_{\alpha}\neq 0\}.\]

We consider the following class of affine varieties with torus actions.

\begin{definition}
Let $X$ be a variety and $\TT_X$ be a torus acting on $X$. The $\TT_X$-action is \emph{good} if it is effective and there is a unique closed orbit, which is a point (called the \emph{vertex}).

We say $(X,\Delta; \TT_X)$ is a \emph{log Fano cone} if the action of $\TT_X$ on $X$ is good, $X$ is affine, $(X,\Delta)$ is klt, and $\Delta$ is $\TT_X$-invariant.
\end{definition}

For any log Fano cone $(X,\Delta; \TT_X)$, we define the \emph{Reeb cone} as
\[\mathfrak{t}^+_{X,\RR}\coloneqq \{\xi\in \mathfrak{t}_{X,\RR}\mid \langle\alpha,\xi\rangle>0,~\text{for every } 0\neq \alpha\in \Gamma_X\}.\]
An element in $\mathfrak{t}^+_{X,\RR}$ is called a \emph{Reeb vector}. Any Reeb vector $\xi$ gives a quasi-monomial $\TT_X$-invariant valuation $\mathrm{wt}_{\xi}$ centered at the vertex, defined by
\[\mathrm{wt}_{\xi}\left(\sum r_{\alpha}\right)\coloneqq \min_{r_{\alpha}\neq 0}\langle\alpha, \xi\rangle\]
for $0\neq \sum r_{\alpha}\in \cO(X)$. We set $\hvol_{(X, \Delta)}(\xi)\coloneqq \hvol_{(X, \Delta)}(\mathrm{wt}_{\xi})$, $\vol_{X,x_0}(\xi)\coloneqq \vol_{X,x_0}(\mathrm{wt}_{\xi})$, and $A_{X,\Delta}(\xi)\coloneqq A_{X,\Delta}(\mathrm{wt}_{\xi})$.

A \emph{polarized log Fano cone} $(X,\Delta; \TT_X,\xi_X)$ is a log Fano cone $(X,\Delta; \TT_X)$ together with a Reeb vector $\xi_X$.

We use the following equivalent definition of the K-semistability of log Fano cones; see \cite{li-xu:higher-rank,li-wang-xu:alg-cone}.

\begin{definition}
A polarized log Fano cone $(X,\Delta; \TT_X,\xi_X)$ is \emph{K-semistable} if $\mathrm{wt}_{\xi_X}$ is the minimizer of the normalized volume function, i.e.~$\hvol(x_0, X, \Delta)=\hvol_{(X,\Delta)}(\xi_X)$, where $x_0\in X$ is the vertex.
\end{definition}

A Reeb vector $\xi$ is said to be \emph{quasi-regular} if $\mathrm{wt}_{\xi}$ is a divisorial
valuation; equivalently, $c\cdot \xi\in N_X$ for some $c\in \RR_{>0}$. As explained in \cite[2.13]{xu-zhuang:kaledin}, every quasi-regular Reeb vector $\xi$ generates a one-parameter subgroup $\chi_{\xi}\colon \Gm\to \TT_X$ and $E(\xi)\coloneqq (X\setminus \{x_0\})/\chi_{\xi}(\Gm)$ is a Koll\'ar component over the vertex $x_0\in X$ such that $\mathrm{wt}_{\xi}=c\cdot \mathrm{ord}_{E(\xi)}$ for some $c>0$. Moreover, we can define a $\QQ$-divisor $\Delta_{E(\xi)}$ by 
\begin{equation}\label{eq:def-deltaE}
p_{\xi}^*(K_{E(\xi)}+\Delta_{E(\xi)})= (K_{X}+\Delta)|_{X^{\circ}},
\end{equation}
where $$p_{\xi}\colon  X^{\circ}\coloneqq X\setminus \{x_0\}\longrightarrow (X\setminus \{x_0\})/\chi_{\xi}(\Gm)$$ is the quotient map. By \cite[Section~4]{kollar:seifert-bundle}, $(E(\xi), \Delta_{E(\xi)})$ is a log Fano pair.

We end this subsection with the following lemma.

\begin{lemma}\label{lem:weight-multiple-claim}
Let $f\colon X=\Spec(B)\to Y=\Spec(A)$ be a finite surjective morphism between affine varieties such that $Y$ is normal. Assume that tori $\TT_X$ and $\TT_Y$ act on $X$ and $Y$, respectively, with an isogeny $q\colon \TT_X\to \TT_Y$ such that $f$ is $q$-equivariant. 

\begin{enumerate}
    \item Composition with $q$ induces inclusions
    \[q_*\colon N_X\longhookrightarrow N_Y\]
    and
    \[q^*\colon M_Y\longhookrightarrow M_X.\]
    We also have a linear isomorphism 
    \begin{equation}\label{eq:induce-iso}
q_*\colon \mathfrak{t}_{X,\RR}\longrightarrow \mathfrak{t}_{Y,\RR}
\end{equation}
after tensoring $q_*$ with $\RR$.

\item We have $q^*(\Gamma_Y)\subset \Gamma_X$.

\item For any $0\neq \alpha\in \Gamma_X$, we can find $d\in \ZZ_{>0}$ so that $d\alpha\in q^*(\Gamma_Y)$.

\item If $A_0=\CC$ and the fixed locus $Y^{\TT_Y}$ consists of a single point $y_0$, then $B_0=\CC$ and $X^{\TT_X}$ consists of a single point $x_0=(f^{-1}(y_0))_{\mathrm{red}}$.
\end{enumerate}
\end{lemma}

\begin{proof}
Since $q$ is finite and surjective, part (a) is standard, while part (b) follows directly from the definition and the $q$-equivariance of $f$.

For part (c), we may take $0\neq b\in B_{\alpha}$ and let
\[p_b(z)=z^d+c_1z^{d-1}+\cdots +c_d\in \CC(A)[z]\]
be the monic minimal polynomial of $b$ over $\CC(A)$. Since $b$ is integral over $A$ and $A$ is normal, every coefficient $c_i$ belongs to $A$ by \citestacks{00H7}.

Recall that for any $\lambda\in \TT_X$, we have $\lambda\cdot b=\chi^\alpha(\lambda)b$. Applying $\lambda$ to $p_b(z)$ gives the monic minimal polynomial of $\lambda\cdot b$ over $\CC(A)$. On the other hand, scaling the variable
shows that the monic minimal polynomial of
$\chi^\alpha(\lambda)b$ is $\chi^{d\alpha}(\lambda)
 p_b\bigl(\chi^{-\alpha}(\lambda)z\bigr)$. By uniqueness of the monic minimal polynomial, we get
\begin{align*}
 z^d+(\lambda\cdot c_1)z^{d-1}+\cdots+(\lambda\cdot c_d)
 &=\chi^{d\alpha}(\lambda)
   p_b\bigl(\chi^{-\alpha}(\lambda)z\bigr)\\
 &=z^d+\chi^\alpha(\lambda)c_1z^{d-1}
   +\cdots+\chi^{d\alpha}(\lambda)c_d.
\end{align*}
Thus every nonzero $c_i$ is $\TT_X$-homogeneous of weight $i\alpha$. Note that the constant term $c_d$ is nonzero because $b\neq 0$. Since $f$ is $q$-equivariant and $q$ is surjective, we see that $d\alpha\in q^*(\Gamma_Y)$. This proves part (c).

For part (d), we take $b\in B_0$. Then the same argument as above shows that every coefficient of $p_b(z)$ is fixed by the $\TT_X$-action. Since $q$ is surjective, we have $$A^{\TT_X}=A^{\TT_Y}=A_0=\CC$$ by our assumption. This implies that $b$ is algebraic over $\CC$. Since $\CC$ is algebraically closed, we get $b\in \CC$ and $B_0\subset \CC$. The reverse inclusion is clear, and we have $B_0=\CC$. In particular, the action of $\TT_X$ on $X$ has only one closed orbit. By our assumption, $y_0$ is fixed by $\TT_Y$, so $f^{-1}(y_0)$ is $\TT_X$-invariant. Since $\TT_X$ is connected, we see that $f^{-1}(y_0)\subset X^{\TT_X}$ as sets. In particular, there is a fixed point $x_0$ of $\TT_X$, hence it is the unique closed orbit. Therefore, we also have $f^{-1}(y_0)=\{x_0\}$ as sets. This proves part (d).
\end{proof}

\subsection{Stable degeneration}
Let $x\in (X=\Spec(R), \Delta)$ be a klt singularity and $v$ be a minimizer of the normalized volume function.

By \cite[Theorem~1.1]{xu-zhuang:stable-deg}, $X_0= \Spec(\mathrm{gr}_v(R))$ is a normal variety and there exists a $\QQ$-divisor $\Delta_{X_0}$ and a torus $\TT_{X_0}$ such that $(X_0, \Delta_{X_0}; \TT_{X_0})$ is a log Fano cone. Moreover, by \cite[Theorem~1.1]{li-xu:higher-rank}, the minimizer $v$ induces a Reeb vector $\xi_{X_0}$ such that $(X_0, \Delta_{X_0}; \TT_{X_0}, \xi_{X_0})$ is a K-semistable log Fano cone with the vertex $x_0$.

Note that in general, $v$ may have higher rational rank, so there is no canonical choice of special degeneration from $x\in (X, \Delta)$ to $x_0\in (X_0, \Delta_{X_0})$. However, according to \cite[Theorem~4.14]{li-xu:higher-rank} (see also \cite[Theorem~1.6]{xu-zhuang:stable-deg}), we can take a Koll\'ar component $E$ near the minimizer $v$ so that the induced special degeneration has central fiber isomorphic to $x_0\in (X_0, \Delta_{X_0})$. In summary, we get:

\begin{theorem}\label{thm:stab-deg}
Let $v$ be a minimizer of the normalized volume function of a klt singularity $x\in (X=\Spec(R), \Delta)$. Then there exists a special degeneration of $x\in (X, \Delta)$ induced by a Koll\'ar component, such that $X_0\cong \Spec(\mathrm{gr}_v(R))$, $(X_0, \Delta_{X_0}; \TT_{X_0}, \xi_{X_0})$ is a K-semistable log Fano cone, and 
\begin{equation}\label{eq:vol}
\hvol(x, X, \Delta)=\hvol(x_0, X_0, \Delta_{X_0}).
\end{equation}
\end{theorem}

\section{Lifting of degenerations}

In this section, we fix a finite surjective morphism $$f\colon \big(x\in (X, \Delta_X)\big)\longrightarrow \big(y\in (Y, \Delta_Y)\big)$$ between klt singularities such that $K_X+\Delta_X=f^*(K_Y+\Delta_Y)$. Starting from a special degeneration of $y\in (Y, \Delta_Y)$, we will construct a compatible special degeneration of $x\in (X, \Delta_X)$ with nice properties.

\subsection{Lifting of special degenerations}

We first show that up to a finite base change and under a mild assumption, any special degeneration of $y\in (Y, \Delta_Y)$ can be lifted to a special degeneration of $x\in (X, \Delta_X)$.

\begin{proposition}\label{prop:pull-back-weak}
Let $\pi_Y\colon (\cY, \Delta_{\cY}; \sigma_{\cY})\to \bA^1$ be a special degeneration of $y\in (Y, \Delta_Y)$. Assume that the natural inclusion
\[\cO(\cY)\subset \cO(Y)[t,t^{-1}]\]
of graded rings induces an identification $\cO(\cY)_0=\cO(Y)$ in degree $0$, where the gradings are induced by the $\Gm$-action. Then up to a finite base change $$\bA^1\longrightarrow \bA^1,\quad t\longmapsto t^n$$ for some $n\in \ZZ_{>0}$, 
there exists a special degeneration 
$$\pi_X\colon (\cX, \Delta_{\cX}; \sigma_{\cX})\longrightarrow \bA^1$$
of $x\in (X, \Delta_X)$ with a finite surjective $\Gm$-equivariant morphism $g\colon \cX\to \cY$ over $\bA^1$ so that
\begin{enumerate}
    \item under the isomorphism $X\times (\bA^1\setminus \{0\})\cong \cX|_{\bA^1\setminus \{0\}}$, we have $g|_{\bA^1\setminus \{0\}}=f\times \mathrm{id}_{\bA^1\setminus \{0\}}$, 

    \item $\deg(g_0)=\deg(f)$,

    \item $K_{X_0}+\Delta_{X_0}=g_0^*(K_{Y_0}+\Delta_{Y_0})$, and

    \item $\sigma_{\cY}=g\circ \sigma_{\cX}$.
\end{enumerate}
\end{proposition}

\begin{proof}
Let $\cX$ be the normalization of $\cY$ in $\CC(X)(t)$. Then we have an induced flat morphism $\pi_X\colon \cX\to \bA^1$ and a finite surjective morphism $g\colon \cX\to \cY$. Note that the induced $\Gm$-coaction on $\CC(Y)(t)$ fixes $\CC(Y)$, so it extends to a $\Gm$-action on $\CC(X)(t)$. From the functoriality of normalization, we then get a $\Gm$-action on $\cX$ such that $\pi_X$ and $g$ are both $\Gm$-equivariant. By \citestacks{03GV}, part (a) holds. Therefore, we have a natural section $\sigma'_{\cX}$ of $\pi_X|_{\bA^1\setminus \{0\}}$ so that $\sigma_{\cY}|_{\bA^1\setminus \{0\}}=g|_{\bA^1\setminus \{0\}}\circ \sigma'_{\cX}$. Since $g$ is finite, by the valuative criterion, $\sigma'_{\cX}$ uniquely extends to a section $\sigma_{\cX}$ of $\pi_{X}$ such that $\sigma_{\cY}=g\circ \sigma_{\cX}$.

Let $X_{0,1},\ldots, X_{0,m}$ be prime divisors on $\cX$ that dominate $Y_0$, and $\eta_1,\ldots,\eta_m$ be their generic points. Then they are precisely irreducible components of $X_0$. Let $\eta$ be the generic point of $Y_0$ and $e_i$ be the ramification index of each $\cO_{\cY, \eta}\hookrightarrow \cO_{\cX, \eta_i}$. We take $n$ to be any integer such that $e_i\mid n$ for each $i$. By \citestacks{0BRM} and \citestacks{09E7}, after base change along $\bA^1\to \bA^1$, $t\mapsto t^n$ and normalization, we may assume that $e_i=1$ for each $i$ and 
\[\mathrm{div}_{\cX}(t)=\sum X_{0,i}\]
as divisors on $\cX$. In particular, $X_0$ is generically reduced. Since $\cX$ is normal, we see that $X_0$ satisfies $S_1$, hence reduced. 
 
Now, we define a $\QQ$-divisor $\Delta_{\cX}$ on $\cX$ by the formula
\begin{equation}\label{eq:crepant}
K_{\cX}+\Delta_{\cX}+X_0=g^*(K_{\cY}+\Delta_{\cY}+Y_0).
\end{equation}
Since $\mathrm{Supp}(\Delta_{\cY})$ does not contain $Y_0$ and $e_i=1$, the Hurwitz formula implies that $\mathrm{coeff}_D(\Delta_{\cX}+X_0)=1$ for every prime divisor $D$ on $\cX$ that dominates $Y_0$. Since $X_0$ is reduced, we conclude that $\mathrm{Supp}(\Delta_{\cX})$ does not contain any component of $X_0$ as well. In particular, $\Delta_{\cX}$ is the closure of $\Delta_{X}\times (\bA^1\setminus \{0\})$ and is effective. Thus, $(\cX, \Delta_{\cX}+X_0)$ is plt by \eqref{eq:crepant} and \cite[Proposition~5.20]{kollar-mori}. Moreover, from $g^*Y_0=X_0$, we get
\[K_{\cX}+\Delta_{\cX}=g^*(K_{\cY}+\Delta_{\cY}).\]
By \cite[Proposition~2.14(i)]{lazic:note}, since $(\cY, \Delta_{\cY})$ is klt, we see that $\lfloor \Delta_{\cX}\rfloor=0$. Then $$X_0=\lfloor \Delta_{\cX}+X_0 \rfloor$$ is normal by \cite[Theorem~4.16]{Kollar13}.

Now, we claim that $X_0$ is connected. Set $A\coloneqq \cO(\cY)$, $B\coloneqq \cO(\cX)$, $S\coloneqq \cO(X)$, $R\coloneqq \cO(Y)$, and $B'\coloneqq \cO(X_0)=B/(t)$. Then $\cO(\cX|_{\bA^1\setminus \{0\}})=S[t, t^{-1}]$ and all these rings are naturally $\ZZ$-graded. By our assumption, we have $A_0=R$. Since each idempotent of a $\ZZ$-graded reduced ring lies in the degree $0$ part, it suffices to show that $B'_0$ has no nontrivial idempotents.

From the construction, we see that $(S[t,t^{-1}])_0=S$. By the inclusion $B\subset S[t,t^{-1}]$, we have $B_0\subset S$. Conversely, $S$ is integral over $R=A_0\subset A$. Hence, every element of $S\subset \CC(X)(t)$ is integral over $A$ and
therefore belongs to the integral closure $B$. Since the $\Gm$-coaction in $S[t,t^{-1}]$ fixes $S$, we have $S\subset B_0$ and get $B_0=S$. As the ideal $(t)\subset B$ is homogeneous, we obtain
\[B'_0=(B/(t))_0=B_0/(tB)_0.\]
In particular, $B'_0$ is local as $S=B_0$ is. Thus, $B'_0$ has no nontrivial idempotents and we conclude that $X_0$ is connected.

Therefore, $\pi_X\colon (\cX, \Delta_{\cX}; \sigma_{\cX})\to \bA^1$ is a special degeneration of $x\in (X, \Delta_X)$.  Part (c) follows directly from the adjunction. Finally, part (b) can be deduced by applying \citestacks{09E8} to $\cO_{\cY, \eta}\hookrightarrow \cO_{\cX, \eta'}$, where $\eta'$ is the generic point of $X_0$, as the corresponding ramification index is $1$.
\end{proof}

Note that by the construction in Section~\ref{subsec:special-deg}, the above assumption holds whenever $\pi_Y$ is induced by a Koll\'ar component.


\subsection{Lifting of log Fano cone structures}

Now, we discuss the lifting problem for a log Fano cone structure. We start with the following general lemma, which may be known to experts. Our argument below is motivated by the proper version in \cite[Proposition~2.4]{brion:purity-quiver}.

For any connected scheme $X$ of finite type over $\CC$ with a closed point $x\in X$, we denote by $\pi_1^{\mathrm{\acute et}}(X, x)$ the corresponding \'etale fundamental group. If $Y$ is also a connected finite type $\CC$-scheme with a closed point $y\in Y$ and a morphism $f\colon X\to Y$ such that $f(x)=y$, then we denote by
\[f_*\colon \pi_1^{\mathrm{\acute et}}(X, x)\longrightarrow \pi_1^{\mathrm{\acute et}}(Y, y)\]
the induced continuous homomorphism.

\begin{lemma}\label{lem:lift-1}
Let $p\colon V\to U$ be a finite \'etale surjective morphism between connected schemes of finite type over $\CC$. Assume that a torus $\TT$ acts on $U$. Then there exists a torus $\wt{\TT}$ acting on $V$ with an isogeny $q\colon \wt{\TT} \to \TT$ such that $p$ is equivariant with respect to $q$. Moreover, if the action of $\TT$ is effective, then we can choose $\wt{\TT}$ so that the action on $V$ is also effective.
\end{lemma}

\begin{proof}
Let $e\in \TT$ be the identity. Let $u_0\in U$ and $v_0\in V$ be two closed points such that $p(v_0)=u_0$. Write
\[H\coloneqq p_*\pi_1^{\mathrm{\acute et}}(V,v_0)
\subset \pi_1^{\mathrm{\acute et}}(U,u_0),\]
which is an open subgroup of finite index by \cite[Thm.~5.4.2]{sza:galois}. Let
\[o_{u_0}\colon \TT\longrightarrow U,\qquad \lambda\longmapsto \lambda\cdot u_0\]
be the orbit morphism, and set $K\coloneqq o_{u_0,*}^{-1}(H)
\subset \pi_1^{\mathrm{\acute et}}(\TT,e)$, which is an open subgroup of finite index in $\pi_1^{\mathrm{\acute et}}(\TT,e)$ by \citestacks{0BND}. Note that $\pi_1^{\mathrm{\acute et}}(\TT,e)$ is abelian, so $K$ is a normal subgroup. Let $q'\colon \TT'\to \TT$ be the connected pointed Galois cover corresponding to $K$ with the chosen point $e'\in \TT'$ above $e\in \TT$ (cf.~\citestacks{03SF}). By \cite{brion:cover-alg-group}, we know that $\TT'$ is also a torus with identity $e'$ and $q'$ is an isogeny.

Now, we consider a morphism
\[b\colon \TT'\times V\longrightarrow U,\qquad
b(\lambda,v)=q'(\lambda)\cdot p(v).\]
By the K\"unneth formula, we know that $\pi_1^{\mathrm{\acute et}}(\TT'\times V,(e',v_0))$ is generated by $j_*\pi_1^{\mathrm{\acute et}}(\TT',e')$ and $l_*\pi_1^{\mathrm{\acute et}}(V,v_0)$, where
\[j\colon \TT'\times \{v_0\}\longhookrightarrow \TT'\times V,\quad l\colon \{e'\}\times V\longhookrightarrow \TT'\times V.\]
Since $b$ factors through $q'\times p$, we have
\[b_*j_*\pi_1^{\mathrm{\acute et}}(\TT',e')=o_{u_0,*}(q'_*\pi_1^{\mathrm{\acute et}}(\TT',e'))=o_{u_0,*}(K)\subset H\]
and
\[b_*l_*\pi_1^{\mathrm{\acute et}}(V,v_0)=p_*\pi_1^{\mathrm{\acute et}}(V,v_0)=H.\]
Consequently, $b_*(\pi_1^{\mathrm{\acute et}}(\TT'\times V,(e',v_0)))\subset H.$ Applying \cite[Prop.~5.5.5]{sza:galois}, this gives a unique morphism
\[a\colon \TT'\times V\longrightarrow V\]
such that $p\circ a=b$ and  $a(e',v_0)=v_0.$ In particular, 
\begin{equation}\label{eq:q-equi}
p(a(\lambda,v))=q'(\lambda)\cdot p(v).
\end{equation}
By the uniqueness of lifting, it is straightforward to see that $a$ gives an action of $\TT'$ on $V$. From \eqref{eq:q-equi}, we see that $p$ is equivariant with respect to $q'$.

Finally, assume that the action of $\TT$ on $U$ is effective. Let $N\coloneqq \ker(\TT'\to \Aut(V))$. Since $p$ is surjective and equivariant, we see that $q'(N)\subset \ker(\TT\to \Aut(U))$, which is trivial. Thus, $N\subset \ker(q')$ is finite. Therefore, $\wt{\TT}\coloneqq \TT'/N$ is a torus that acts effectively on $V$ and $q'$ factors as an isogeny $q\colon \wt{\TT}\to \TT$. This proves the lemma.
\end{proof}

When $f$ is not necessarily \'etale, we have the following result.

\begin{lemma}\label{lem:lift-2}
Let $f\colon X\to Y$ be a finite surjective morphism between varieties with $X$ normal. Assume that a torus $\TT_Y$ acts on $Y$ and there is a nonempty $\TT_Y$-invariant open subset $U\subset Y$ such that $f^{-1}(U)\to U$ is \'etale. Then there exists a torus $\TT_X$ acting on $X$ and an isogeny $q\colon \TT_X\to \TT_Y$ such that $f$ is equivariant with respect to $q$. If the action of $\TT_Y$ on $Y$ is effective, then we can take such $\TT_X$ so that its action on $X$ is also effective. 
\end{lemma}

\begin{proof}
By Lemma~\ref{lem:lift-1}, there is an isogeny $q\colon \TT_X\to \TT_Y$ of tori such that $\TT_X$ acts on $f^{-1}(U)$ and $f|_{f^{-1}(U)}\colon f^{-1}(U)\to U$ is equivariant with respect to $q$.

Let $a_Y\colon \TT_Y\times Y\to Y$ be the action morphism, and define
\[h\colon \TT_X\times X\longrightarrow Y,
 \qquad
 h(\lambda,x)\coloneqq a_Y(q(\lambda),f(x)).\]
Consider the fiber product 
\[\begin{tikzcd}
	P & X \\
	{\TT_X\times X} & Y.
	\arrow[from=1-1, to=1-2]
	\arrow[from=1-1, to=2-1]
	\arrow["f", from=1-2, to=2-2]
	\arrow["h", from=2-1, to=2-2]
\end{tikzcd}\] 
Then the first projection $p_1\colon P\to \TT_X\times X$ is finite. Since $f|_{f^{-1}(U)}$ is equivariant with respect to $q$, the action constructed in Lemma~\ref{lem:lift-1} gives a section
of $p_1|_{p_1^{-1}(\TT_X\times f^{-1}(U))}$. Let $\Gamma$ be the closure of its image in $P$ with the reduced scheme structure. Then from the construction, the induced morphism $\Gamma\to \TT_X\times X$ is finite and birational, therefore is an isomorphism by the normality of $\TT_X\times X$ and \citestacks{0AB1}. Composing its inverse with the second projection $P\to X$ gives a morphism
\begin{equation}
\label{eq:extended-X-action}
 a_X\colon \TT_X\times X\longrightarrow X
\end{equation}
which extends the action on $f^{-1}(U)$ and satisfies
\begin{equation}
\label{eq:g-equivariant-global}
 f(a_X(\lambda,x))=a_Y(q(\lambda),f(x)).
\end{equation}
The identity and associativity axioms hold over the dense open subsets $\TT_X\times f^{-1}(U)$ and $\TT_X\times \TT_X\times f^{-1}(U)$, respectively, so they hold everywhere. Thus, \eqref{eq:extended-X-action} is an action of $\TT_X$ on $X$, and \eqref{eq:g-equivariant-global} implies that $f$ is $q$-equivariant.

Finally, if the action of $\TT_Y$ on $Y$ is effective, then so is the action on $U$, and the same holds for the action of $\TT_X$ on $f^{-1}(U)$ by Lemma~\ref{lem:lift-1}. Thus, the last statement follows.
\end{proof}

The next result shows that the log Fano cone structure can be lifted along a finite log-crepant morphism.

\begin{proposition}\label{prop:lift-cone}
Let $(Y, \Delta_{Y}; \TT_Y)$ be a log Fano cone with the vertex $y_0$. Let $(X, \Delta_X)$ be a klt pair with a finite surjective morphism $f\colon X\to Y$ so that $K_X+\Delta_X=f^*(K_Y+\Delta_Y)$. Then there exists a torus $\TT_X$ such that $(X, \Delta_{X}; \TT_X)$ is a log Fano cone and the following hold:

\begin{enumerate}
    \item there exists an isogeny $q\colon \TT_X\to \TT_Y$ and $f$ is equivariant with respect to $q$, 

    \item $(f^{-1}(y_0))_{\mathrm{red}}$ is a single point and is the vertex of $(X, \Delta_{X}; \TT_X)$, and

    \item $q_*(\mathfrak{t}^+_{X,\RR})=\mathfrak{t}^+_{Y,\RR}$.
\end{enumerate}

\end{proposition}

\begin{proof}
By the assumption and purity of branch locus, we know that $f$ is \'etale over a nonempty open subset $U\coloneqq Y_{\mathrm{reg}}\setminus \mathrm{Supp}(\Delta_Y)$. Moreover, since $\Delta_Y$ is $\TT_Y$-invariant, we see that $U$ is $\TT_Y$-invariant as well. Then Lemma~\ref{lem:lift-2} gives a torus $\TT_X$ that acts effectively on $X$ with an isogeny $q\colon \TT_X\to \TT_Y$ such that $f$ is $q$-equivariant. By Lemma~\ref{lem:weight-multiple-claim}(d), we know that the action of $\TT_X$ on $X$ is good with the vertex $x_0=(f^{-1}(y_0))_{\mathrm{red}}$. Then $(X, \Delta_{X}; \TT_X)$ is a log Fano cone and part (b) follows once we know that $\Delta_X$ is $\TT_X$-invariant.

If $E\subset X$ is a prime divisor and $D\coloneqq f(E)$, then from the assumption $K_X+\Delta_X=f^*(K_Y+\Delta_Y)$, we have 
\begin{equation}
\label{eq:coefficient-boundary-under-finite-map}
 \coeff_E(\Delta_X)
 =e_{E/D}\coeff_D(\Delta_Y)-(e_{E/D}-1),
\end{equation}
where $e_{E/D}$ is the ramification index of $E$ over $D$. For any $\lambda\in \TT_X$, $e_{E/D}=e_{\lambda\cdot E/q(\lambda)\cdot D}$. Since $\Delta_Y$ is $\TT_Y$-invariant, we have $\coeff_D(\Delta_Y)=\coeff_{q(\lambda)\cdot D}(\Delta_Y)$. Then by \eqref{eq:coefficient-boundary-under-finite-map}, we get 
\[\coeff_{\lambda\cdot E}(\Delta_X)=e_{\lambda\cdot E/q(\lambda)\cdot D}\coeff_{q(\lambda)\cdot D}(\Delta_Y)-(e_{\lambda\cdot E/q(\lambda)\cdot D}-1)=\coeff_E(\Delta_X).\]
This proves the $\TT_X$-invariance of $\Delta_X$.


It remains to prove part (c). By \eqref{eq:induce-iso}, composition with $q$ gives a linear isomorphism
\[q_*\colon \mathfrak{t}_{X,\RR}\longrightarrow \mathfrak{t}_{Y,\RR},\]
which is characterized by the pairings:
\begin{equation}
\label{eq:pairing-compatibility-q}
 \langle q^*\alpha,\xi\rangle
 =\langle\alpha,q_*\xi\rangle
\end{equation}
for any $\alpha\in M_Y$ and $\xi\in\mathfrak t_{X,\mathbb R}$. Let $\xi\in \mathfrak t^+_{X,\mathbb R}$. If
$\alpha\in\Gamma_Y\setminus\{0\}$, then
$q^*\alpha\in\Gamma_X\setminus\{0\}$ by Lemma~\ref{lem:weight-multiple-claim}(b), which gives $\langle\alpha,q_*\xi\rangle
 =\langle q^*\alpha,\xi\rangle>0$. Therefore, we get $q_*(\mathfrak t^+_{X,\mathbb R})\subseteq\mathfrak t_{Y,\RR}^+.$ 
 
Conversely, let $\eta\in\mathfrak t_{Y,\RR}^+$ and set $\xi\coloneqq q_*^{-1}(\eta).$ For any $\beta\in\Gamma_X\setminus\{0\}$,
Lemma~\ref{lem:weight-multiple-claim}(c) gives an integer $m>0$ and $\alpha\in\Gamma_Y\setminus\{0\}$ such that $m\beta=q^*\alpha$. Hence
\[m\langle\beta,\xi\rangle
 =\langle q^*\alpha,\xi\rangle
 =\langle\alpha,q_*\xi\rangle
 =\langle\alpha,\eta\rangle
 >0.\]
Therefore, we get $\xi\in\mathfrak t^+_{X,\mathbb R}$ and  the inclusion $\mathfrak t_{Y,\RR}^+\subseteq q_*(\mathfrak t^+_{X,\mathbb R})$ follows. Combining the above two inclusions gives $q_*(\mathfrak t^+_{X,\mathbb R})=\mathfrak t_{Y,\RR}^+$, which finishes the proof.
\end{proof}

\section{Local K-stability under finite morphisms}

In this section, we prove the following result, which will be used in the next section to analyze normalized volumes of central fibers in stable degenerations.

\begin{theorem}\label{thm:Kss}
Let $(X,\Delta_X; \TT_X)$ and $(Y, \Delta_Y; \TT_Y)$ be log Fano cones with an isogeny $q\colon \TT_X\to \TT_Y$. Let $f\colon (X,\Delta_X)\to (Y, \Delta_Y)$ be a finite surjective $q$-equivariant morphism satisfying $K_X+\Delta_X=f^*(K_Y+\Delta_Y)$. If $\xi_X\in \mathfrak{t}^+_{X,\RR}$ and $\xi_Y\in \mathfrak{t}^+_{Y,\RR}$ are Reeb vectors such that $q_*(\xi_X)=\xi_Y$, then $(X,\Delta_X; \TT_X, \xi_X)$ is K-semistable if and only if $(Y,\Delta_Y; \TT_Y, \xi_Y)$ is K-semistable.
\end{theorem}

In the global setting, the corresponding result for log Fano pairs is proved in \cite{liu-zhu:equivariant}. To prove this theorem, we will approximate $\xi_X$ and $\xi_Y$ by nearby quasi-regular Reeb vectors, and then apply \cite[Theorem~1.2(3)]{liu-zhu:equivariant} to the corresponding Koll\'ar components.

To relate the local stability to the global stability of Koll\'ar components, we need the following invariant in the local setting.

\begin{definition}
Let $(X,\Delta_X; \TT_X)$ be a log Fano cone with the vertex $x_0$ and write $R\coloneqq \cO(X)$. Fix $\xi \in \mathfrak{t}^+_{X,\RR}$ and $v\in \Val_{X,x_0}^{\TT_X}$ with $A_{X,\Delta_X}(v)<+\infty$. For any $m\in \RR_{>0}$, we define 
\[\wt{S}_m(\xi;v)\coloneqq \sum_{\substack{\lambda\in \RR_{\geq 0}\\ \alpha\in \Gamma_X\\ \langle \alpha,\xi\rangle <m}} \lambda\cdot \dim_{\CC} \mathrm{gr}^{\lambda}_v R_{\alpha},\]
where $$\mathrm{gr}_v(R)=\bigoplus_{\substack{ \lambda\in \RR_{\geq 0}\\ \alpha\in \Gamma_X}}\mathrm{gr}^{\lambda}_v R_{\alpha}$$
is the induced weight decomposition on $\mathrm{gr}_v(R)$ (cf.~\cite[Section~3.1]{xu-zhuang:unique}).

We then define
\[\wt{S}(\xi;v)\coloneqq \lim_{m\to +\infty} \frac{(\dim X+1)!}{m^{\dim X+1}}\wt{S}_m(\xi;v)\]
and
\[S(\xi;v)\coloneqq \frac{A_{X,\Delta_X}(\xi)}{\wt{S}(\xi;\mathrm{wt}_{\xi})}\wt{S}(\xi;v).\]
\end{definition}

Note that $\wt{S}(\xi;v)$ is well-defined and positive by \cite{xu-zhuang:unique}. By definition, we have
\begin{equation}\label{eq:S-homogeneity}
 \wt{S}(c\xi;v)=c^{-(\dim X+1)}\wt{S}(\xi;v),
 \qquad
 S(c\xi;v)=S(\xi;v)
\end{equation}
for any $c>0$.

The invariant $S(\xi;v)$ characterizes the K-stability of a log Fano cone.

\begin{theorem}\label{thm:xz-216}
Let $(X,\Delta_X; \TT_X, \xi)$ be a polarized log Fano cone with the vertex $x_0$. Then it is K-semistable if and only if
\[A_{X,\Delta_X}(v)\geq S(\xi; v)\]
for all quasi-monomial $v\in \Val_{X,x_0}^{\TT_X}$.
\end{theorem}

\begin{proof}
This is a combination of \cite[Theorem~3.10]{xu-zhuang:unique} and \cite[Theorem~1.3]{li-xu:higher-rank}.
\end{proof}

For a log Fano pair $(E, \Delta_E)$, its K-semistability is controlled by the $\delta$-invariant $\delta(E,\Delta_E)$ (cf.~\cite{fujita,blum-jonsson:threshold}). The relation between $S(\xi;v)$ and the $\delta$-invariant of Koll\'ar components near $\xi$ is recorded in the next two lemmas.

\begin{lemma}[{\cite[Lemma~2.18]{xu-zhuang:kaledin}}]\label{lem:xz-218}
Let $(X,\Delta_X; \TT_X)$ be a log Fano cone with the vertex $x_0$ and let $\xi \in \mathfrak{t}^+_{X,\RR}$ be a quasi-regular Reeb vector. Then for any $\epsilon\geq 0$, $$\delta(E(\xi), \Delta_{E(\xi)})\geq 1-\epsilon$$ if and only if
\[A_{X,\Delta_X}(v)\geq (1-\epsilon)S(\xi; v)\]
for all quasi-monomial $v\in \Val_{X,x_0}^{\TT_X}$.
\end{lemma}

\begin{lemma}[{\cite[Lemma~2.19]{xu-zhuang:kaledin}}]\label{lem:xz-219}
Let $(X,\Delta_X; \TT_X, \xi)$ be a K-semistable log Fano cone. Then for any $\epsilon>0$, there exists an open neighborhood $U\subset \mathfrak{t}^+_{X,\RR}$ of $\xi$ such that for any quasi-regular $\xi'\in U$, we have
\[\delta(E(\xi'), \Delta_{E(\xi')})\geq 1-\epsilon.\]
\end{lemma}

We also need the following continuity result.

\begin{lemma}\label{lem:continue-S}
Let $(X,\Delta_X; \TT_X)$ be a log Fano cone with the vertex $x_0$. Then for any $v\in \Val^{\TT_X}_{X,x_0}$ with $A_{X,\Delta_X}(v)<+\infty$, the function
\[\xi\longmapsto S(\xi;v)\]
is continuous on $\mathfrak{t}^+_{X,\RR}$.
\end{lemma}

\begin{proof}
This follows from an argument similar to the proof of \cite[Lemma~2.19]{xu-zhuang:kaledin}. Set $n\coloneqq \dim X$.

Let $\sigma\subseteq M_{X,\RR}$ be the weight cone, namely the cone generated by $\Gamma_X$. Since the $\TT_X$-action is good, $\sigma$ is a strongly convex closed polyhedral cone, and a Reeb vector is strictly positive on $\sigma\setminus\{0\}$. Therefore, if we fix $\xi_0\in\mathfrak t^+_{X,\RR}$, the slice
\[P_{\xi_0}\coloneqq 
\{\alpha\in\sigma\mid \langle\alpha,\xi_0\rangle=1\}\]
is closed and bounded by \cite[Theorem~8.4]{convex}, hence compact.  For $\xi$ sufficiently close to $\xi_0$, put
\[c_-(\xi)\coloneqq 
 \min_{\alpha\in P_{\xi_0}}\langle\alpha,\xi\rangle,
 \qquad
 c_+(\xi)
 \coloneqq 
 \max_{\alpha\in P_{\xi_0}}\langle\alpha,\xi\rangle.\]
Then $c_-(\xi),c_+(\xi)>0$, and
\[
\lim_{\xi\to \xi_0} c_{\pm}(\xi)=1.\]
Moreover, by homogeneity in $\alpha$, we get
\begin{equation}\label{eq:weight-comparison}
 c_-(\xi)\langle\alpha,\xi_0\rangle
 \leq
 \langle\alpha,\xi\rangle
 \leq
 c_+(\xi)\langle\alpha,\xi_0\rangle
\end{equation}
for any $\alpha\in\sigma$.

The right-hand inequality in \eqref{eq:weight-comparison} implies
\[\{\alpha\in \Gamma_X\mid
   \langle\alpha,c_+(\xi)\xi_0\rangle<m\}
 \subseteq
 \{\alpha\in \Gamma_X\mid
   \langle\alpha,\xi\rangle<m\},\]
whereas the left-hand inequality implies
\[\{\alpha\in \Gamma_X\mid
   \langle\alpha,\xi\rangle<m\}
 \subseteq
 \{\alpha\in \Gamma_X\mid
   \langle\alpha,c_-(\xi)\xi_0\rangle<m\}.\]
So
\[\wt{S}_m(c_+(\xi)\xi_0;v)
 \leq
 \wt{S}_m(\xi;v)
 \leq
 \wt{S}_m(c_-(\xi)\xi_0;v).\]
After dividing by $m^{n+1}/(n+1)!$, passing to the limit, and using
\eqref{eq:S-homogeneity}, we obtain
\begin{equation*}
 c_+(\xi)^{-(n+1)}\wt{S}(\xi_0;v) \leq \wt{S}(\xi;v) \leq c_-(\xi)^{-(n+1)}\wt{S}(\xi_0;v).
\end{equation*}
Thus, $\wt{S}(\xi;v)$ is continuous at $\xi_0$.

Finally, $\xi\mapsto A_{X,\Delta_X}(\xi)$ is continuous on the Reeb cone by \cite[Theorem~2.15(3)]{li-xu:higher-rank}, while
\[\wt{S}(\xi;\mathrm{wt}_{\xi})=n\cdot \vol_{X,x_0}(\mathrm{wt}_\xi)\]
by \cite[(3.4) and pp.~160]{xu-zhuang:unique}, which is continuous
on the Reeb cone as well by \cite[Proposition~3.10]{li-xu:higher-rank} (see also \cite[Theorem~3]{cs:irregular-sasakian}). Therefore, it follows from the definition that $S(\xi;v)$ is continuous on the Reeb cone.
\end{proof}

\begin{lemma}\label{lem:crepant-kollar}
Let $(X,\Delta_X; \TT_X)$ and $(Y, \Delta_Y; \TT_Y)$ be log Fano cones and $q\colon \TT_X\to \TT_Y$ be an isogeny. Let $f\colon (X,\Delta_X)\to (Y, \Delta_Y)$ be a finite surjective $q$-equivariant morphism satisfying $K_X+\Delta_X=f^*(K_Y+\Delta_Y)$. If $\xi_X\in \mathfrak{t}^+_{X,\RR}$ and $\xi_Y\in \mathfrak{t}^+_{Y,\RR}$ are quasi-regular Reeb vectors such that $q_*(\xi_X)=\xi_Y$, then we have a finite surjective morphism
\[h\colon E(\xi_X)\longrightarrow E(\xi_Y)\]
such that $K_{E(\xi_X)}+\Delta_{E(\xi_X)}=h^*(K_{E(\xi_Y)}+\Delta_{E(\xi_Y)})$.
\end{lemma}

\begin{proof}
We follow the notation in Section~\ref{subsec:cone}. By assumption, we see that the morphism $f^{\circ}\coloneqq f|_{X^{\circ}}\colon X^{\circ}\to Y^{\circ}$ is finite surjective satisfying 
\begin{equation}\label{eq:fcirc}
(K_{X}+\Delta_{X})|_{X^{\circ}}=(f^{\circ})^*((K_{Y}+\Delta_{Y})|_{Y^{\circ}}).
\end{equation}
Let $\chi_{\xi_X}\colon \Gm\to \TT_X$ and  $\chi_{\xi_Y}\colon \Gm\to \TT_Y$ be the one-parameter subgroups generated by $\xi_X$ and $\xi_Y$, respectively. From our assumption, we see that $q \circ \chi_{\xi_X}=\chi_{\xi_Y}\circ [k]$ for some $k\in \ZZ_{>0}$, where $[k]\colon \Gm\to \Gm$ is defined by $t\mapsto t^k$. In particular, by the universal property of quotients, we have a morphism $h$ and a commutative diagram 
\[\begin{tikzcd}
	{X^{\circ}} & {Y^{\circ}} \\
	{E(\xi_X)} & {E(\xi_Y)}
	\arrow["{f^{\circ}}", from=1-1, to=1-2]
	\arrow["{p_{\xi_X}}"', from=1-1, to=2-1]
	\arrow["{p_{\xi_Y}}", from=1-2, to=2-2]
	\arrow["h", from=2-1, to=2-2]
\end{tikzcd}\]
so that $h$ is surjective. By \eqref{eq:def-deltaE} and \eqref{eq:fcirc}, we see that $$K_{E(\xi_X)}+\Delta_{E(\xi_X)}=h^*(K_{E(\xi_Y)}+\Delta_{E(\xi_Y)}).$$ Since $-(K_{E(\xi_X)}+\Delta_{E(\xi_X)})$ and $-(K_{E(\xi_Y)}+\Delta_{E(\xi_Y)})$ are both ample, it follows that $h$ is finite.
\end{proof}

\begin{proof}[{Proof of Theorem~\ref{thm:Kss}}]
Let $x_0\in X$ and $y_0\in Y$ be vertices. First, assume that $(Y,\Delta_Y; \TT_Y, \xi_Y)$ is K-semistable. If $(X,\Delta_X; \TT_X, \xi_X)$ is not K-semistable, then by Theorem~\ref{thm:xz-216}, there exists a quasi-monomial valuation $v\in \Val^{\TT_X}_{X,x_0}$ such that 
\[A_{X, \Delta_X}(v)<S(\xi_X; v).\]
Thus, we can find $\frac{1}{2}>\epsilon>0$ sufficiently small such that
\[A_{X, \Delta_X}(v)\leq (1-2\epsilon)S(\xi_X; v).\]
Therefore, Lemma~\ref{lem:continue-S} gives an open neighborhood $U_{\epsilon}\subset \mathfrak{t}^+_{X,\RR}$ of $\xi_X$ so that
\[A_{X, \Delta_X}(v)< (1-\epsilon)S(\xi; v)\]
for any $\xi\in U_{\epsilon}$. By Lemma~\ref{lem:xz-218}, for any quasi-regular $\xi\in U_{\epsilon}$, we have
\[\delta(E(\xi), \Delta_{E(\xi)})<1-\epsilon.\]
Combining this with Lemma~\ref{lem:crepant-kollar} and \cite[Theorem 1.2(3)]{liu-zhu:equivariant}, we obtain
\[\delta(E(\xi'), \Delta_{E(\xi')})<1-\epsilon\]
for any quasi-regular $\xi'\in q_*(U_{\epsilon})$. But this contradicts Lemma~\ref{lem:xz-219} as $\xi_Y\in q_*(U_{\epsilon})$ and quasi-regular Reeb vectors are dense.

The remaining implication is completely the same, after switching the roles of $X$ and $Y$ in the above argument.
\end{proof}

\section{Finite degree formula}

In this section, we prove Theorem~\ref{thm:main} and Corollary~\ref{cor:global}. Using the above results, the proof of Theorem~\ref{thm:main} will be reduced to the following lemma.

\begin{lemma}\label{lem:vol-cone}
Let $(X,\Delta_X; \TT_X)$ and $(Y, \Delta_Y; \TT_Y)$ be log Fano cones and $q\colon \TT_X\to \TT_Y$ be an isogeny. Let $f\colon (X,\Delta_X)\to (Y, \Delta_Y)$ be a finite surjective $q$-equivariant morphism satisfying $K_X+\Delta_X=f^*(K_Y+\Delta_Y)$. If $\xi_X\in \mathfrak{t}^+_{X,\RR}$ and $\xi_Y\in \mathfrak{t}^+_{Y,\RR}$ are Reeb vectors such that $q_*(\xi_X)=\xi_Y$, then 
\[\hvol_{(X, \Delta_X)}(\xi_X)=\deg(f)\cdot \hvol_{(Y, \Delta_Y)}(\xi_Y).\]
\end{lemma}

\begin{proof}
Let $x_0\in X$ and $y_0\in Y$ be the vertices and write $X=\Spec(B)$ and $Y=\Spec(A)$. We first claim that $A_{X, \Delta_X}(\xi_X)=A_{Y, \Delta_Y}(\xi_Y)$. Since the log discrepancy function is linear on Reeb cones, it suffices to assume that $\xi_X$ and $\xi_Y$ are both quasi-regular, which is equivalent to the divisoriality of $\mathrm{wt}_{\xi_X}$ and $\mathrm{wt}_{\xi_Y}$. Moreover, from the construction, we have
\begin{align*}
\mathrm{wt}_{\xi_X}|_{\CC(Y)}(r)=\mathrm{wt}_{\xi_X}(r)&=\min_{r_{\alpha}\neq 0}\langle q^*\alpha,\xi_X \rangle\\
&= \min_{r_{\alpha}\neq 0}\langle \alpha,q_*(\xi_X) \rangle\\
&= \min_{r_{\alpha}\neq 0}\langle \alpha,\xi_Y \rangle\\
&=\mathrm{wt}_{\xi_Y}(r)
\end{align*}
for any $r\in A$. Hence, $\mathrm{wt}_{\xi_X}|_{\CC(Y)}=\mathrm{wt}_{\xi_Y}$ and the claim follows from \cite[Proposition~2.14]{lazic:note}. It remains to show 
\[\vol_{X, x_0}(\xi_X)=\deg(f)\cdot \vol_{Y, y_0}(\xi_Y).\]

To this end, set $n=\dim X=\dim Y$. For any finite $M_X$-graded $A$-module $M$, define
\[P_M(z)\coloneqq \sum_{\substack{\alpha\in M_X\\\langle \alpha,\xi_X \rangle<z}} \dim_{\CC} M_{\alpha}.\]
Note that any $M_Y$-graded $A$-module can be regarded as an $M_X$-graded $A$-module via the embedding $q^*\colon M_Y\hookrightarrow M_X$. In particular, we have the $M_X$-grading on $A$ given by $A_{q^*\alpha}\coloneqq A_{\alpha}$ and
\[A=\bigoplus_{\beta \in q^*(\Gamma_Y)} A_{\beta}.\]
Then from Definition~\ref{def:volume} and $q_*(\xi_X)=\xi_Y$, we have
\begin{equation}\label{eq:hilb-poly}
P_A(z)=\frac{\vol_{Y, y_0}(\xi_Y)}{n!}z^n+o(z^n),\quad P_B(z)=\frac{\vol_{X, x_0}(\xi_X)}{n!}z^n+o(z^n).
\end{equation}

Now, choose homogeneous elements $b_i\in B_{\chi_i}$ for $1\leq i \leq \deg(f)$ forming a $\CC(Y)$-basis of $\CC(X)$, where $\chi_i\in \Gamma_X$. Then we get an injective graded homomorphism
\[
 \bigoplus_{i=1}^{\deg(f)} A(-\chi_i)\longhookrightarrow B
\]
with a torsion cokernel $Q$. Here, $A(-\chi_i)$ is $M_X$-graded as
\[(A(-\chi_i))_{\alpha}\coloneqq A_{\alpha-\chi_i}\]
for any $\alpha\in M_X$. Therefore, we obtain
\[P_B(z)=P_Q(z)+\sum_{i=1}^{\deg(f)} P_A(z-\langle \chi_i, \xi_X\rangle).\]
Note that by \eqref{eq:hilb-poly}, we still have
\[P_A(z-\langle \chi_i, \xi_X\rangle)=\frac{\vol_{Y, y_0}(\xi_Y)}{n!}z^n+o(z^n),\]
so we obtain
\[\frac{\vol_{X, x_0}(\xi_X)}{n!}z^n+o(z^n)=\frac{\deg(f)\cdot \vol_{Y, y_0}(\xi_Y)}{n!}z^n+P_Q(z)+o(z^n).\]
Since $\dim(\mathrm{Supp} (Q))\le n-1$, we have $P_Q(z)=o(z^n)$, which completes the proof.
\end{proof}

Now, we are ready to prove our main theorem.

\begin{proof}[{Proof of Theorem~\ref{thm:main}}]
By Theorem~\ref{thm:stab-deg}, we have a special degeneration $(\cY, \Delta_{\cY}; \sigma_{\cY})\to \bA^1$ of the klt singularity $y\in (Y, \Delta_Y)$ such that $(Y_0, \Delta_{Y_0})$ carries the structure of a K-semistable log Fano cone $(Y_0, \Delta_{Y_0}; \TT_{Y_0}, \xi_{Y_0})$ with the vertex $y_0=\sigma_{\cY}(0)$. Applying Proposition~\ref{prop:pull-back-weak}, we may assume that there exists a special degeneration $(\cX, \Delta_{\cX}; \sigma_{\cX})\to \bA^1$ with a $\Gm$-equivariant finite surjection $g\colon \cX\to \cY$ and $g_0(x_0)=y_0$, where $x_0=\sigma_{\cX}(0)$. Applying Proposition~\ref{prop:lift-cone}, $(X_0,\Delta_{X_0})$ also admits a log Fano cone structure $(X_0,\Delta_{X_0}; \TT_{X_0})$ with the vertex $x_0$ and an isogeny $q\colon \TT_{X_0}\to \TT_{Y_0}$ such that $g_0$ is $q$-equivariant.

If we set $\xi_{X_0}\coloneqq (q_*)^{-1}(\xi_{Y_0})$, then by Theorem~\ref{thm:Kss}, $(X_0,\Delta_{X_0}; \TT_{X_0}, \xi_{X_0})$ is K-semistable. Therefore, Lemma~\ref{lem:vol-cone} gives
\begin{align*}
\hvol(x_0, X_0, \Delta_{X_0})=&\hvol_{(X_0, \Delta_{X_0})}(\xi_{X_0})\\
=&\deg(g_0)\cdot \hvol_{(Y_0, \Delta_{Y_0})}(\xi_{Y_0})\\
=&\deg(g_0)\cdot \hvol(y_0, Y_0, \Delta_{Y_0}).
\end{align*}
By \eqref{eq:vol} and Proposition~\ref{prop:pull-back-weak}(b), we get
\[\hvol(x_0, X_0, \Delta_{X_0})= \deg(f)\cdot \hvol(y, Y, \Delta_{Y}).\]
Using \cite[Theorem~1]{blum-liu:lower}, we obtain
\[\hvol(x, X, \Delta_{X})\geq \deg(f)\cdot \hvol(y, Y, \Delta_{Y}).\]

On the other hand, for any divisorial valuation $v\in \Val_{Y, y}$, take $w\in \Val_{X,x}$ such that $w|_{\CC(Y)}=v$. Then $A_{Y,\Delta_Y}(v)=A_{X,\Delta_X}(w)$ follows from \cite[Proposition 5.20]{kollar-mori}. Moreover, we have
\begin{equation}\label{eq:extend-a}
\mathfrak{a}_m(v)\cO(X)\subset \mathfrak{a}_m(w)
\end{equation}
for any $m\geq 0$. It follows that
\[\mathrm{length}(\cO(X)/\mathfrak{a}_m(w))\leq \mathrm{length}(\cO(X)/\mathfrak{a}_m(v)\cO(X)).\]
By choosing a basis of $\CC(X)$ over $\CC(Y)$ and clearing the denominators, we get an injection $\cO(Y)^{\oplus \deg(f)}\hookrightarrow \cO(X)$ with a cokernel $C$, which also implies
\[\mathrm{length}(\cO(X)/\mathfrak{a}_m(v)\cO(X))\leq \deg(f)\cdot \mathrm{length}(\cO(Y)/\mathfrak{a}_m(v))+\mathrm{length}(C/\mathfrak{a}_m(v)C).\]
Combining this with \eqref{eq:extend-a}, we obtain
\[\vol_{X,x}(w)\leq \deg(f)\cdot \vol_{Y,y}(v)+\limsup_{m\to +\infty} \frac{(\dim X)!}{m^{\dim X}}\mathrm{length}(C/\mathfrak{a}_m(v)C).\]
Since $\dim \mathrm{Supp}(C)\leq \dim X-1$, we conclude that $\vol_{X,x}(w)\leq \deg(f)\cdot \vol_{Y,y}(v)$, hence
\[\hvol(x, X, \Delta_{X})\leq\hvol_{(X,\Delta_X)}(w)\leq \deg(f)\cdot \hvol_{(Y,\Delta_Y)}(v)\]
follows from $A_{Y,\Delta_Y}(v)=A_{X,\Delta_X}(w)$. Taking the infimum over all divisorial $v$ and using \eqref{eq:hvol-divisor}, we get
\[\hvol(x, X, \Delta_{X})\leq \deg(f)\cdot \hvol(y, Y, \Delta_{Y}),\]
which completes the proof.
\end{proof}

\begin{proof}[{Proof of Corollary~\ref{cor:global}}]
Using \citestacks{02LN}, we may assume that $f^{-1}(\{y\})=\{x\}$ as sets. In this case, the induced morphism $$f_x\colon \big(x\in (\Spec(\cO_{X,x}), \Delta_X|_{\Spec(\cO_{X,x})})\big)\longrightarrow \big(y\in (\Spec(\cO_{Y,y}),\Delta_Y|_{\Spec(\cO_{Y,y})})\big)$$
is finite by \citestacks{05B8}. Now, the result follows from applying Theorem~\ref{thm:main} to $f_x$.
\end{proof}


\bibliography{formula}

\bibliographystyle{alpha}

\end{document}